\documentclass[preprint,12pt,3p]{elsarticle}

\usepackage{amssymb}

\usepackage{amsthm}
\usepackage{graphicx}
\usepackage{amsmath}
\usepackage{hyperref}
\usepackage{graphicx}

\usepackage{pgf,tikz}

\usetikzlibrary{snakes}
\tikzstyle{printersafe}=[snake=snake,segment amplitude=0 pt]

\usetikzlibrary{arrows}
\usetikzlibrary{shapes}
\usepackage{multirow}
\usepackage{latexsym}
\usepackage{setspace}
\PassOptionsToPackage{boxed,section}{algorithm}
\usepackage{algorithm}
\usepackage{algorithmic}
\usepackage{enumerate}
\usepackage{amsmath}			
\usepackage{amssymb}			
\usepackage{url}
\usepackage{setspace}
\usepackage{tikz}
\usetikzlibrary{arrows}
\usetikzlibrary{shapes}
\usetikzlibrary{decorations.markings}
\usepackage{geometry}
\usepackage{bbm}

\newtheorem{problem}{\em Problem}
\newtheorem{theorem}{\em Theorem}
\newtheorem{conjecture}{\em Conjecture}

\newtheorem{lemma}{\em Lemma}

\newtheorem{observation}{\em Observation}

\journal{Latex Templates}

\begin{document}

\begin{frontmatter}

\title{An NP-hardness result for the colored constrained maximum 2-edge-colorable subgraph problem in bipartite graphs}

\author[label1]{Vahan Mkrtchyan}
\address[label1]{Department of Mathematics and Computer Science,\\ College of the Holy Cross, Worcester, MA, USA}

%
\ead{vahan.mkrtchyan@gssi.it}
%
%

\begin{abstract}
In this paper, we consider the maximum $k$-edge-colorable subgraph problem. In this problem we are given a graph $G$ and a positive integer $k$, the goal is to take $k$ matchings of $G$ such that their union contains maximum number of edges. This problem is NP-hard in cubic graphs, and polynomial-time solvable in bipartite graphs as we observe in our paper. We present an NP-hardness result for a version of this problem where we have color constraints on vertices. In fact, we show that this version is NP-hard already in bipartite graphs of maximum degree three. In order to achieve the result, we establish a connection between our problem and the problem of construction of special maximum matchings considered in the Master thesis of the author and defended back in 2003.  
\end{abstract}


\end{frontmatter}



\section{Introduction}
\label{IntroSection}

In this paper, we consider finite, undirected graphs without loops or parallel edges. The set of vertices and edges of a graph $G$ is denoted by $V(G)$ and $E(G)$, respectively. The degree of a vertex $v$ of $G$ is denoted by $d_{G}(v)$. Let $\Delta(G)$ and $\delta(G)$ be the maximum and minimum degree of a vertex of $G$. A graph $G$ is regular, if $\delta(G)=\Delta(G)$. A graph is cubic if $\delta(G)=\Delta(G)=3$. The girth of the graph is the length of the shortest cycle in it. A matching in a graph $G$ is a subset of edges such that no vertex of $G$ is incident to two edges from the subset. A maximum matching is a matching that contains maximum possible number of edges. A maximum matching is perfect if every vertex of the graph is incident to an edge from the perfect matching. If $M$ is a matching of a graph $G$, then an $M$-augmenting path of $G$ is a simple path of odd length, such that the edges with odd indices lie outside $M$, and the edges with even indices belong to $M$. 

A graph is bipartite if its set of vertices can be divided into two disjoint sets $U_1$ and $U_2$, such that every edge connects a vertex in $U_1$ to one in $U_2$. A graph is called a forest if it does not contain a cycle. Note that a forest can be disconnected. In a special case when it is connected, the graph is called a tree.

If $k\geq 0$, then a graph $G$ is called {$k$-edge colorable}, if its edges can be assigned colors from a set of $k$ colors so that adjacent
edges receive different colors. The smallest integer $k$, such that $G$ is $k$-edge-colorable is called chromatic index of $G$ and is denoted by $\chi'(G)$. The classical theorem of Shannon states that if $G$ is a multi-graph then $\Delta(G)\leq \chi'(G) \leq \left \lfloor \frac{3\Delta(G)}{2} \right \rfloor$ \cite{Shannon:1949,stiebitz:2012}. On the other hand, the classical theorem of Vizing states that $\Delta(G)\leq \chi'(G) \leq \Delta(G)+\mu(G)$ for any multi-graph $G$ \cite{stiebitz:2012,vizing:1964}. Here $\mu(G)$ is the maximum multiplicity of an edge of $G$. A graph is class I if $\chi'(G)=\Delta(G)$, otherwise it is class II.


If $k<\chi'(G)$, we cannot color all edges of $G$ with $k$ colors. Thus it is reasonable to investigate the maximum number of edges that one can color with $k$ colors. A subgraph $H$ of a graph $G$ is called {maximum $k$-edge-colorable}, if $H$ is $k$-edge-colorable and contains the maximum number of edges among all $k$-edge-colorable subgraphs. For $k\geq 0$ and a graph $G$ let
\[\nu_{k}(G) = \max \{ |E(H)| : H \text{ is a $k$-edge-colorable subgraph of } G \}. \]
Clearly, a $k$-edge-colorable subgraph is maximum if it contains exactly $\nu_k(G)$ edges. Note that $\nu_1(G)$ is the size of a maximum matching in $G$. Usually, we will shorten the notation $\nu_1(G)$ to $\nu(G)$.

From the first glance it may seem that if we have a maximum $k$-edge-colorable subgraph of a graph, then by adding some edges to it, we can get a maximum $(k+1)$-edge-colorable subgraph. The example from Figure \ref{fig:Examplenu1nu2} shows that this is not true. It has a unique perfect matching, which contains the edge joining the two degree-three vertices. However, the unique maximum $2$-edge-colorable subgraph of it contains all its eight edges except the edge joining the two degree-three vertices.

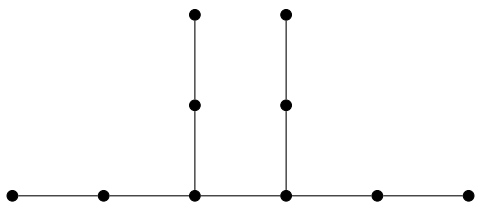
\begin{figure}[ht]
\centering
  
  \begin{center}

		\begin{tikzpicture}[scale = 0.6]
			
			
			
			

			\tikzstyle{every node}=[circle, draw, fill=black!50,inner sep=0pt, minimum width=4pt]
																							
			\node[circle,fill=black,draw] at (0,0) (n00) {};
			
			\node[circle,fill=black,draw] at (-2, 0) (nm20) {};
																								
			\node[circle,fill=black,draw] at (-4,0) (nm40) {};

                \node[circle,fill=black,draw] at (0,2) (n02) {};

                \node[circle,fill=black,draw] at (0,4) (n04) {};
                

                \node[circle,fill=black,draw] at (2,0) (n20) {};

                \node[circle,fill=black,draw] at (4,0) (n40) {};

                \node[circle,fill=black,draw] at (6,0) (n60) {};

                \node[circle,fill=black,draw] at (2,2) (n22) {};

                \node[circle,fill=black,draw] at (2,4) (n24) {};

			


			\path[every node]
			
			(n00) edge (nm20)
                (nm20) edge (nm40)

                (n00) edge (n02)
                (n02) edge (n04)

                (n00) edge (n20)
   

                (n20) edge (n40)
                (n40) edge (n60)
			
			(n20) edge (n22)
                (n22) edge (n24)

			;
		\end{tikzpicture}
																
	\end{center}
								
	\caption{A graph in which the maximum matching is not a subset of a maximum $2$-edge-colorable subgraph.}
	\label{fig:Examplenu1nu2}
\end{figure}

There are several papers where the ratio $\frac{|E(H_k)|}{|E(G)|}$ has been investigated. Here $H_k$ is a maximum $k$-edge-colorable subgraph of $G$. In \cite{bollobas:1978,henning:2007,nishizeki:1981,nishizeki:1979,weinstein:1974} lower bounds are proved for the ratio when the graph is regular and $k=1$. For regular graphs of high girth the bounds are improved in \cite{flaxman:2007}. Albertson and Haas have investigated the problem in \cite{haas:1996,haas:1997} when $G$ is a cubic graph. See also \cite{samvel:2010,corrigendum}, where the authors proved that for every cubic graph $G$ $\nu _{2}(G) \geq \frac{4}{5}|V(G)|$ and $\nu _{3}(G) \geq  \frac{7}{6} |V(G)|$. Moreover, \cite{samvel:2014} shows that for any cubic graph $G$ $\nu _{2}(G) + \nu _{3}(G) \geq 2|V(G)|$.

    
    Bridgeless cubic graphs that are not $3$-edge-colorable are usually called snarks \cite{cavi:1998}, and the problem for snarks is investigated by Steffen in \cite{steffen:1998,steffen:2004}. This lower bound has also been investigated in the case when the graphs need not be cubic in \cite{miXumbFranciaciq:2013,Kaminski:2014,Rizzi:2009}. Kosowski and Rizzi have investigated the problem from the algorithmic perspective \cite{Kosowski:2009,Rizzi:2009}, see also \cite{FeigeOfekWieder}. Since the problem of constructing a $k$-edge-colorable graph in an input graph is NP-complete for each fixed $k\geq 2$, it is natural to investigate the (polynomial) approximability of the problem. In \cite{Kosowski:2009}, for each $k\geq 2$ an algorithm for the problem is presented. There for each fixed value of $k \geq 2$, algorithms are proved to have certain approximation ratios and they are tending to $1$ as $k$ tends to infinity. See the paper, \cite{AKS2022,AINA2024,JGAA2024} for recent results on this problem.

Some structural properties of maximum $k$-edge-colorable subgraphs of graphs are proved in \cite{samvel:2014,MkSteffen:2012}. In particular, there it is shown that every set of disjoint cycles of a graph with $\Delta=\Delta(G) \geq 3$ can be extended to a maximum $\Delta$-edge colorable subgraph.  Also there it is shown that a maximum $\Delta$-edge colorable subgraph of a simple graph is always class I. Finally, if $G$ is a graph with girth $g \in \left \{ 2k, 2k+1 \right \} (k \geq 1)$ and $H$ is a maximum $\Delta$-edge colorable subgraph of $G$, then $\frac{|E(H)|}{|E(G)|} \geq \frac{2k}{2k+1}$ and the bound is best possible is a sense that there is an example attaining it. See \cite{Cao:2023} for recent results that deal with partitioning any graph into two class I subgraphs.

In \cite{samvel:2010} Mkrtchyan et al. proved that $\nu_{2}(G) \leq \frac{|V(G)| + 2\nu_{3}(G)}{4}$ for any cubic graph $G$. For bridgeless cubic graphs, which by Petersen theorem have a perfect matching \cite{Lovasz}, this inequality becomes, $\nu _{2}(G)\leq \frac{\nu _{1}(G)+\nu _{3}(G)}{2}$. One may wonder whether there are other interesting graph-classes, where a relation between $\nu_2(G)$ and $\frac{\nu_1(G)+\nu_3(G)}{2}$ can be proved. The main result of \cite{LianaDAM:2019} states that for a given integer $k\geq $0 and a finite bipartite multi-graph $G$, without loops, the following inequality holds
\begin{equation*}
\nu_k(G)\geq \frac{\nu_{k-i}(G)+\nu_{k+i}(G)}{2},
\end{equation*}
for $i=0,1,...,k$. Note that \cite{LianaDAM:2019} predicts that
\begin{conjecture}(\cite{LianaDAM:2019})
    Given an integer $k\geq $0 and a finite multi-graph $G$, without loops, the following inequality holds
\begin{equation*}
\nu_k(G)\geq \frac{\nu_{k-i}(G)+\nu_{k+i}(G)-b(G)}{2},
\end{equation*}
for $i=0,1,...,k$. Here $b(G)$ denotes the smallest number of vertices of $G$ whose removal leads to a bipartite graph.
\end{conjecture} For partial results towards this conjecture see \cite{BASM2023}.

In this paper, we deal with the exact solvability of the maximum $k$-edge-colorable subgraph problem. Its precise formulation is the following:
\begin{problem}\label{prob:MaxkEdgeColsub}
	 Given a graph $G$ and an integer $k$, find a $k$-edge-colorable subgraph with maximum number of edges together with its $k$-edge-coloring.
\end{problem} 

In this paper, we focus on the maximum 2-edge-colorable subgraph problem which is the restriction of the problem to the case $k=2$. As we observe in our paper this problem is NP-hard in cubic graphs. Moreover, using the theory of maximum flows in networks, we show that it is polynomial-time solvable in bipartite graphs. The main result of the paper is an NP-hardness result for a version of this problem where we have color constraints on vertices. Our main result implies that this version is hard already in connected, bipartite graphs of maximum degree three. The main technical contribution of the paper is the establishment of a connection between our problems and the problem of construction of special maximum matchings considered in the Master thesis of the author and defended back in 2003. For the notions, facts and concepts that are not explained in the paper the reader is referred to the graph theory monograph of West \cite{west:1996}.



\section{Main results}
\label{MainSection}

In this section we obtain the main results of the paper. We start with the following observation:

\begin{observation}
    \label{obs:Cubics} The problem of computing $\nu_2(G)$ is NP-hard in cubic graphs.
\end{observation}

\begin{proof} Let $G$ be a cubic graph. Note that $G$ is 3-edge-colorable, if and only if $G$ contains a pair of edge-disjoint perfect matchings. This conditions is equivalent to $\nu_2(G)=|V|$. Since testing the 3-edge-colorability of a cubic graph is an NP-complete problem \cite{holyer:1981}, we get the result.
\end{proof}

Our next observation states that there is a polynomial algorithm to compute $\nu_k(G)$ in arbitrary (not necessarily cubic) bipartite graphs $G$ and $k\geq 0$.

\begin{observation}
    \label{obs:Bips} The problem of computing $\nu_k(G)$ is polynomial-time solvable in bipartite graphs.
\end{observation}

\begin{proof} We borrow ideas from \cite{CSIT2005}. Let $G=(A, B, E)$ be a bipartite graph with a bipartition $V=A\cup B$. Consider a network $H$ obtained from $G$ as follows (Figure \ref{fig:nukBipartiteNetworks}):
\begin{itemize}
    \item add a source $s$ joined to every vertex of $A$ with an arc of capacity $k$,

    \item orient every edge $e=uv\in E$ with $u\in A$ and $v\in B$ from $u$ to $v$ and assign it a capacity $1$,

    \item add a sink $t$ such that every vertex of $B$ is joined to $t$ with an arc of capacity $k$.
\end{itemize}

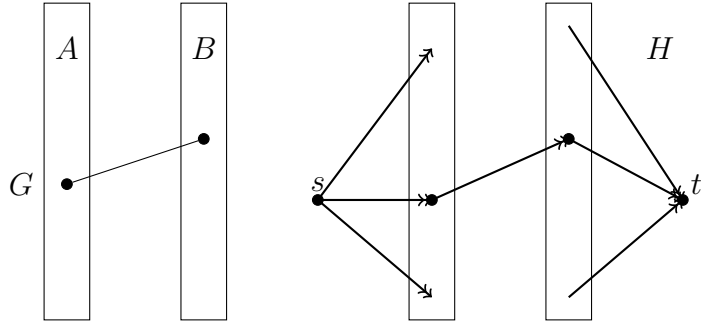
\begin{figure}[ht]
\centering
  
  \begin{center}

		\begin{tikzpicture}[scale = 0.6]
			
			
			
			

                \draw[black] (0,0) rectangle (-1,7);

                \draw[black] (2,0) rectangle (3,7);

                \node at (-1.5,3) {$G$};
                \node at (-0.5,6) {$A$};
                \node at (2.5,6) {$B$};

                \node at (5,3) {$s$};
                \node at (13.3,3) {$t$};

                 \node at (12.5,6) {$H$};

                \draw[black] (7,0) rectangle (8,7);
                \draw[black] (10,0) rectangle (11,7);

                \draw[->>, thick] (7.5,2.65) -- (10.5,4);
                \draw[->>, thick] (5,2.65) -- (7.5,2.65);
			\draw[->>, thick] (10.5,4)--(13,2.65);

                \draw[->>, thick] (5,2.65)--(7.5,6);
                \draw[->>, thick] (5,2.65)--(7.5,0.5);
                \draw[->>, thick] (10.5,6.5)--(13,2.65);
                \draw[->>, thick] (10.5,0.5)--(13,2.65);
			
			\tikzstyle{every node}=[circle, draw, fill=black!50,inner sep=0pt, minimum width=4pt]

                \node[circle,fill=black,draw] at (5,2.65) (ns) {};
                \node[circle,fill=black,draw] at (7.5,2.65) (nu) {};
																							
			\node[circle,fill=black,draw] at (-0.5,3) (nm13) {};

                \node[circle,fill=black,draw] at (2.5,4) (n24) {};
                \node[circle,fill=black,draw] at (10.5,4) (nv) {};

                \node[circle,fill=black,draw] at (13,2.65) (nt) {};



			\path[every node]
			
			(nm13) edge (n24)

			;
		\end{tikzpicture}
																
	\end{center}
								
	\caption{The network $H$ obtained from the bipartite graph $G$. The source $s$ of $H$ is joined to every vertex of $A$ with an arc of capacity $k$. Every edge of $G$ is replaced with an arc of capacity $1$. Every vertex of $B$ is joined to the sink $t$ with an arc of capacity $k$.}
	\label{fig:nukBipartiteNetworks}
\end{figure}

Let $f_{max}(H)$ denote the value of the maximum flow in $H$. We claim that
\[f_{max}(H)=\nu_k(G).\]
For the proof of this statement, let us start with an arbitrary $k$-edge-colorable subgraph $I_k$ of $G$. Define a function $f$ on edges of $H$ as follows:
\begin{itemize}
    \item for every edge $e\in E(G)$, $f(e)=1$ if $e\in I_k$, and $f(e)=0$ if $e\notin I_k$; 

    \item for every directed edge $e=su$, $u\in A$ of $H$ define $f(e)=d_{I_k}(u)$;

    \item for every directed edge $e=vt$, $v\in B$ of $H$ define $f(e)=d_{I_k}(v)$.
\end{itemize} It can be easily seen that $f$ is a flow in $H$ whose value is $|E(I_k)|$. Hence,
\[f_{max}(H)\geq \nu_k(G).\]
Now, for the proof of the converse inequality, consider a maximum flow $f_m$ in $H$. Note that since capacities in $H$ are integral (they are either 1 or $k$), via standard results in flow theory (see \cite{Lovasz}) we can assume that $f_m$ takes integral values. Consider a spanning subgraph $J_k$ of $G$ obtained from $f_m$ as follows: we take the edge $e=uv$, $u\in A$, $v\in B$ in $J_k$ if $f_m(e)=1$. Note that since the directed edges incident to $s$ or $t$ are of capacity $k$, we have that all vertices in $J_k$ are of degree at most $k$. By the classical theorem of K\"{o}nig (see \cite{Lovasz}), $J_k$ is $k$-edge-colorable. Moreover, the value of $f_m$ is equal to $|E(J_k)|$. Hence
\[f_{max}(H)=|E(J_k)|\leq \nu_k(G).\]
\end{proof}

Observation \ref{obs:Bips} is also shown in \cite{Rugby2021} (see Theorem 1 (b)). There the considered problem is called Multi-STC with two colors but it is discussed that it is equivalent to 2-edge-colorable subgraph on triangle-free graphs and the NP-hardness is shown via triangle-free graphs.

If $H_0\subseteq G$ is a subgraph of $G$ and $p:E(G)\rightarrow \mathbb{N}$ is an edge-weight function, then the $p$-weight of $H_0$ is defined as
\[p(H_0)=p(E(H_0))=\sum_{e\in E(H_0)}p(e).\]
The following lemma is from \cite{JGAA2024}. Its proof uses ideas that are present already in \cite{AINA2021}.

\begin{lemma}
\label{lem:ForestConstraintsWeightedKmatchings} (\cite{AINA2021,JGAA2024}) Let $k\geq 1$ and $G$ be an edge-weighted forest with $p:E\rightarrow \mathbb{N}$. Suppose $W:V\rightarrow 2^{\{1,...,k\}}$ is a function that assigns each vertex $u$ a subset $W(u)\subseteq \{1,...,k\}$ of admissible colors. Then, there is an $O((k+1)\cdot 2^{2k}\cdot |V|)$-time algorithm that finds a largest weighted $k$-edge-colorable subgraph (with respect to $p$) with the constraint that around every vertex $v$ {\bf only} colors from $\{\mathit{0}\}\cup W(v)$ appear.
\end{lemma} Note that in this lemma the running time of the algorithm has exponential dependence on $k$. Moreover, we can have multiple edges of color $0$ around a vertex $u$. However, the colors from $W(u)$ can appear at most once around it.

\medskip

In this paper, we will be interested in the following question: can we generalize Lemma \ref{lem:ForestConstraintsWeightedKmatchings} to arbitrary bipartite graphs? The question is legitimate and interesting because of Observation \ref{obs:Bips}.

The main result of this paper implies that under the assumption $P\neq NP$, our question has a negative answer. We obtain our hardness result for the following case:
\begin{itemize}
    \item $k=2$, there are no weights on edges ($w(e)=1$ for every edge $e\in E$), however there are subsets $W(u)\subseteq \{1,...,k\}$ of admissible colors around every vertex $u\in V$.
\end{itemize}

Note that if there are weights on edges, however there are no subsets $W(u)\subseteq \{1,...,k\}$ of admissible colors around every vertex $u\in V$, that is, $W(u)=\{1,...,k\}$ for every vertex $u\in V$ and $k$ is arbitrary, this version of our problem is polynomial-time solvable by the classical result of Gabow \cite{Gabow83}.

Rather surprisingly, in order to obtain our result, we establish a connection between our problem and the Master thesis of the author that was defended in 2003 (see \cite{KM2008} for its journal version). Since we will need some details from \cite{KM2008}, let us present them first. \cite{KM2008} deals with the following question. Let $G$ be a graph. Define:
\[\ell(G)=\min\{\nu(G\backslash F): F\text{ is a maximum matching of }G\},\]
and
\[L(G)=\max\{\nu(G\backslash F): F\text{ is a maximum matching of }G\}.\]

Note that if $G$ is the path of length four (Figure \ref{fig:Example4path}), then $\ell(G)=1$ and $L(G)=2$. In \cite{KM2008}, the problem of computing $\ell(G)$ and $L(G)$ is considered. The main contributions of the paper are that both of these parameters are NP-hard to compute in the connected, bipartite graphs of maximum degree three.

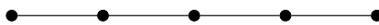
\begin{figure}[ht]
\centering
  
  \begin{center}

		\begin{tikzpicture}[scale = 0.6]
			
			
			
			

			\tikzstyle{every node}=[circle, draw, fill=black!50,inner sep=0pt, minimum width=4pt]
																							
			\node[circle,fill=black,draw] at (0,0) (n00) {};

                \node[circle,fill=black,draw] at (2,0) (n20) {};

                \node[circle,fill=black,draw] at (4,0) (n40) {};

                \node[circle,fill=black,draw] at (6,0) (n60) {};

                \node[circle,fill=black,draw] at (8,0) (n80) {};



			\path[every node]
			
			(n00) edge (n20)
                (n20) edge (n40)
                (n40) edge (n60)
                (n60) edge (n80)

			;
		\end{tikzpicture}
																
	\end{center}
								
	\caption{In the path of length four, $\ell(G)=1$ and $L(G)=2$.}
	\label{fig:Example4path}
\end{figure}

The reductions presented in \cite{KM2008} are from Max $2$-SAT that we define below.
\begin{problem}
    \label{prob:Max2SAT} (Max $2$-SAT) On the input we are given $m$ clauses $C_1, ..., C_m$ each of which containing two literals of boolean variables $x_1,..., x_n$, and a number $K\leq m$. The goal is to check whether there is a truth assignment $\Tilde{\alpha}=(\alpha_1,...,\alpha_n)$, such that at least $K$ of clauses $C_1, ..., C_m$ are satisfied by $\Tilde{\alpha}$.
\end{problem} Though $2$-SAT is polynomial-time solvable \cite{Pap1994}, Max $2$-SAT is NP-complete \cite{GJ,Max2SAT}. Like it is done in \cite{KM2008}, we will assume that every boolean variable $x_i$ appears in at least two clauses $C_j$. 

%

\begin{figure}[ht]
\centering
\begin{minipage}[b]{.5\textwidth}
  \begin{center}
	\begin{tikzpicture}[scale = 0.6]
			
			
			
			

                \draw[->, thick] (-1, 0) -- (10,0);
                \draw[->, thick] (0, -1) -- (0, 10);

                \node at (5,-1) {$4i-1$};
                \node at (7,-1) {$4i$};

                \node at (-1,3) {$4j-3$};
                \node at (-1,5) {$4j-2$};
                \node at (-1,7) {$4j-1$};
                \node at (-1,9) {$4j$};

                \node at (5.65,3) {$u_{11}$};
                \node at (7.65,3) {$u_{21}$};
                \node at (5.65,5) {$u_{12}$};
                \node at (7.65,5) {$u_{22}$};

                \node at (7.65,7) {$v_{22}$};
                \node at (7.65,9) {$v_{12}$};
                \node at (5.65,6.65) {$v_{21}$};
                \node at (5.65,8.65) {$v_{11}$};

                \draw [dashed] (5,3) -- (5,0);
                \draw [dashed] (7,3) -- (7,0);

                \draw [dashed] (5,3) -- (0,3);
                \draw [dashed] (5,5) -- (0,5);
                \draw [dashed] (5,7) -- (0,7);
                \draw [dashed] (5,9) -- (0,9);

			\tikzstyle{every node}=[circle, draw, fill=black!50,inner sep=0pt, minimum width=4pt]
																							
			\node[circle,fill=black,draw] at (5,3) (u11) {};
                \node[circle,fill=black,draw] at (5,5) (u12) {};
                \node[circle,fill=black,draw] at (7,3) (u21) {};
                \node[circle,fill=black,draw] at (7,5) (u22) {};

                \node[circle,fill=black,draw] at (5,9) (v11) {};
                \node[circle,fill=black,draw] at (7,9) (v12) {};
                \node[circle,fill=black,draw] at (5,7) (v21) {};
                \node[circle,fill=black,draw] at (7,7) (v22) {};



			\path[every node]
			
			(u11) edge (u12)
                (u21) edge (u22)
                (u12) edge (v21)
                (u22) edge (v22)

                (v21) edge (v22)
                (v22) edge (v12)
                (v11) edge (v12)

			;
		\end{tikzpicture}

	\end{center}
	
	\caption{The gadget corresponding to the variable\\ $x_i$ and the clause $C_j$.}\label{fig:Gadgeta}
\end{minipage}%
\begin{minipage}[b]{.5\textwidth}
  	\begin{center}
	\centering
  
  \begin{center}

		\begin{tikzpicture}[scale = 0.6]
			
			
			
			

                \draw[->, thick] (-1, 0) -- (10,0);
                \draw[->, thick] (0, -1) -- (0, 10);

                \node at (2.65,-1) {$4i-3$};
                \node at (5,-1) {$4i-2$};
                 \node at (7.35,-1) {$4i-1$};
                \node at (9.35,-1) {$4i$};

                \node at (-1,5) {$4j-1$};
                \node at (-1,7) {$4j$};

                \node at (3,5.35) {$u_{11}$};
                \node at (3, 7.35) {$u_{21}$};
                \node at (5,5.35) {$u_{12}$};
                \node at (5,7.35) {$u_{22}$};

                \node at (10,5.35) {$v_{22}$};
                \node at (9.65,7.35) {$v_{12}$};
                \node at (7.65,5.35) {$v_{21}$};
                \node at (7.65,7.35) {$v_{11}$};

                \draw [dashed] (2.65,5) -- (2.65,0);
                \draw [dashed] (5,5) -- (5,0);
                \draw [dashed] (7.35,5) -- (7.35,0);
                \draw [dashed] (9.35,5) -- (9.35,0);

                 \draw [dashed] (2.65,5) -- (0,5);
                 \draw [dashed] (2.65,7) -- (0,7);

			\tikzstyle{every node}=[circle, draw, fill=black!50,inner sep=0pt, minimum width=4pt]
																							
			\node[circle,fill=black,draw] at (2.65,5) (u11) {};
                \node[circle,fill=black,draw] at (5,5) (u12) {};
                \node[circle,fill=black,draw] at (2.65,7) (u21) {};
                \node[circle,fill=black,draw] at (5,7) (u22) {};

                \node[circle,fill=black,draw] at (7.35,7) (v11) {};
                \node[circle,fill=black,draw] at (9.35,7) (v12) {};
                \node[circle,fill=black,draw] at (7.35,5) (v21) {};
                \node[circle,fill=black,draw] at (9.35,5) (v22) {};



			\path[every node]
			
			(u11) edge (u12)
                (u12) edge (v21)
                (v21) edge (v22)
                (v22) edge (v12)
                (v11) edge (v12)
                (v11) edge (u22)
                (u21) edge (u22)

			;
		\end{tikzpicture}
																
	\end{center}
								
	\caption{The gadget corresponding to the literal $\overline{x}_i$ and clause $C_j$.}
	\label{fig:Gadgetb}
	\end{center}
\end{minipage}
\end{figure}

Now we are going to describe a graph $G_I$ constructed from an instance $I=(X, C, K)$ of Max $2$-SAT in \cite{KM2008}. Our graph is going to have vertices which will be integral points on the plane. In other words, our vertices will be pairs $(x,y)$ where both $x$ and $y$ are integers. 

Suppose $x_i$ ($1\leq i\leq n$) is a boolean variable from $X$ appearing in a clause $C_j$ ($1\leq j\leq m$) from $C$. If $x_i$ appears as a variable in $C_j$, then the graph corresponding to it is from Figure \ref{fig:Gadgeta}, and if $x_i$ appears as a negated variable in $C_j$, then the graph corresponding to it is from Figure \ref{fig:Gadgetb}. This graph is going to appear as a part of a larger graph. Sometimes, we will prefer not to draw the vertices $u_{ij}$ ($1\leq i,j \leq 2$). Thus, we will use a conventional sign for them. This sign is the one from Figure \ref{fig:ConventionalSign}.

\begin{figure}[ht]
\centering
  
  \begin{center}

		\begin{tikzpicture}[scale = 0.6]
			
			
			
			

                \node at (0,0.5) {$v_{11}$};
                \node at (2,0.5) {$v_{12}$};
                \node at (0,-1.5) {$v_{21}$};
                \node at (2.5,-1.65) {$v_{22}$};

			\tikzstyle{every node}=[circle, draw, fill=black!50,inner sep=0pt, minimum width=4pt]
																							
			\node[circle,fill=black,draw] at (0,0) (v11) {};
                \node[circle,fill=black,draw] at (2,0) (v12) {};
                \node[circle,fill=black,draw] at (2,-2) (v22) {};
                \node[circle,fill=black,draw] at (0,-2) (v21) {};



			\path[every node]
			
			(v11) edge (v12)
                (v12) edge (v22)
                (v22) edge (v21)

			;
		\end{tikzpicture}
																
	\end{center}
								
	\caption{The conventional sign.}
	\label{fig:ConventionalSign}
\end{figure}
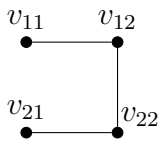

The vertices $v_{ij}$ ($1\leq i, j\leq 2$) are the ones from the corresponding graph. Now, if $C_j=t_{j_1}\vee t_{j_2}$  ($1\leq j_1 < j_2\leq m$) is a clause containing literals $t_{j_1}$, $t_{j_2}$ of variables $x_{j_1}$ and $x_{j_2}$, then the graph $G(C_j)$ corresponding to $C_j$ is the one from Figure \ref{fig:GraphCorrespondingClause}.

\begin{figure}[ht]
\centering
  
  \begin{center}

		\begin{tikzpicture}[scale = 0.6]
			
			
			
			


                \draw[->, thick] (-1, 0) -- (10,0);
                \draw[->, thick] (0, -1) -- (0, 10);

                \node at (2,-1) {$4j_1-1$};
                \node at (4,-1) {$4j_1$};

                \node at (7,-1) {$4j_2-1$};
                \node at (9,-1) {$4j_2$};

                \node at (-1.5,7) {$4j-1$};
                \node at (-1,9) {$4j$};


                \draw [dashed] (2,7) -- (2,0);
                \draw [dashed] (4,7) -- (4,0);
                \draw [dashed] (7,7) -- (7,0);
                \draw [dashed] (9,7) -- (9,0);

                \draw [dashed] (2,7) -- (0,7);
                \draw [dashed] (2,9) -- (0,9);

			\tikzstyle{every node}=[circle, draw, fill=black!50,inner sep=0pt, minimum width=4pt]
																							
			\node[circle,fill=black,draw] at (2,7) (v27) {};
                \node[circle,fill=black,draw] at (4,7) (v47) {};
                \node[circle,fill=black,draw] at (2,9) (v29) {};
                \node[circle,fill=black,draw] at (4,9) (v49) {};

                \node[circle,fill=black,draw] at (7,7) (v77) {};
                \node[circle,fill=black,draw] at (7,9) (v79) {};
                \node[circle,fill=black,draw] at (9,7) (v97) {};
                \node[circle,fill=black,draw] at (9,9) (v99) {};

                \node[circle,fill=black,draw] at (0,7) (v07) {};
                \node[circle,fill=black,draw] at (0,9) (v09) {};



			\path[every node]
			
			(v27) edge (v47)
                (v47) edge (v49)
                (v49) edge (v29)

                (v77) edge (v97)
                (v97) edge (v99)
                (v99) edge (v79)

                (v07) edge (v09)

                (v09) edge [bend right] (v07)
                (v07) edge (v49)
                (v07) edge (v99)

			;
		\end{tikzpicture}
																
	\end{center}
								
	\caption{The graph $G(C_j)$ corresponding to the clause $C_j$.}
	\label{fig:GraphCorrespondingClause}
\end{figure}
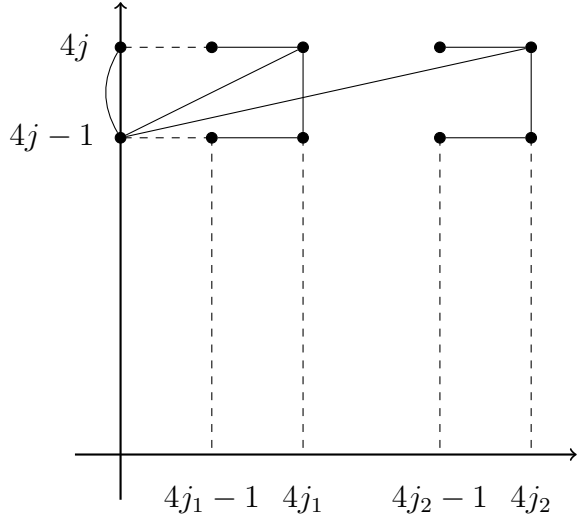

Now, for $i=1,...,n$ let $C_{j_1}, ..., C_{j_{{r(i)}}}$ ($j_1<j_2<...<j_{j_{{r(i)}}}$) be the clauses containing a literal of $x_i$ ($r(i)\geq 2$). Define a graph $G(I)$ corresponding to $I$ as follows: if $G(C_1)$, ..., $G(C_m)$ are the graphs corresponding to clauses $C_1,..., C_m$, then for $i=1,...,n$ cyclically connect $G(C_{j_1}), ..., G(C_{j_{{r(i)}}})$ Figure \ref{fig:GraphCorrespondingVariable}.
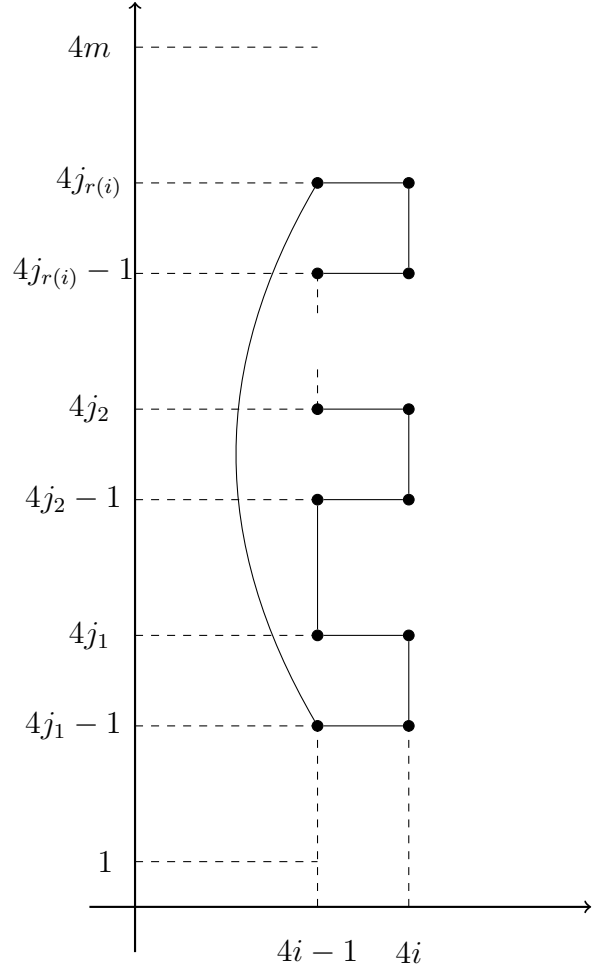
\begin{figure}[ht]
\centering
  
  \begin{center}

		\begin{tikzpicture}[scale = 0.6]
			
			
			
			


                \draw[->, thick] (-1, 0) -- (10,0);
                \draw[->, thick] (0, -1) -- (0,20);

                \node at (4,-1) {$4i-1$};
                \node at (6,-1) {$4i$};

                \node at (-0.65,1) {$1$};

                \node at (-1.35,4) {$4j_1-1$};
                \node at (-1,6) {$4j_1$};

                \node at (-1.35,9) {$4j_2-1$};
                \node at (-1,11) {$4j_2$};

                \node at (-1.35,14) {$4j_{r(i)}-1$};
                \node at (-1,16) {$4j_{r(i)}$};

                \node at (-1,19) {$4m$};


                \draw [dashed] (4,11) -- (4,12);
                \draw [dashed] (4,14) -- (4,13);

                \draw [dashed] (4,0) -- (4,4);
                \draw [dashed] (6,0) -- (6,4);

                \draw [dashed] (0,1) -- (4,1);
                \draw [dashed] (0,4) -- (4,4);
                \draw [dashed] (0,6) -- (4,6);
                \draw [dashed] (0,9) -- (4,9);
                \draw [dashed] (0,11) -- (4,11);
                \draw [dashed] (0,14) -- (4,14);
                \draw [dashed] (0,16) -- (4,16);
                \draw [dashed] (0,19) -- (4,19);

			\tikzstyle{every node}=[circle, draw, fill=black!50,inner sep=0pt, minimum width=4pt]
																							
			\node[circle,fill=black,draw] at (4,4) (v44) {};
                \node[circle,fill=black,draw] at (6,4) (v64) {};
                \node[circle,fill=black,draw] at (6,6) (v66) {};
                \node[circle,fill=black,draw] at (4,6) (v46) {};

                \node[circle,fill=black,draw] at (4,9) (v49) {};
                \node[circle,fill=black,draw] at (6,9) (v69) {};
                \node[circle,fill=black,draw] at (6,11) (v611) {};
                \node[circle,fill=black,draw] at (4,11) (v411) {};

                 \node[circle,fill=black,draw] at (4,14) (v414) {};
                \node[circle,fill=black,draw] at (6,14) (v614) {};
                \node[circle,fill=black,draw] at (6,16) (v616) {};
                \node[circle,fill=black,draw] at (4,16) (v416) {};



			\path[every node]
			
			(v44) edge (v64)
                (v64) edge (v66)
                (v66) edge (v46)

                (v46) edge (v49)

                (v49) edge (v69)
                (v69) edge (v611)
                (v611) edge (v411)

                (v414) edge (v614)
                (v614) edge (v616)
                (v616) edge (v416)

                (v44) edge [bend left] (v416)

			;
		\end{tikzpicture}
																
	\end{center}
								
	\caption{Cyclically joining the subgraphs that correspond to the variable $x_i$.}
	\label{fig:GraphCorrespondingVariable}
\end{figure}

As it is stated in \cite{KM2008}, the constructed graph $G(I)$ may not be connected. Therefore in order to obtain a connected one let us consider a graph $G_I$ constructed from $G(I)$ as it is stated in Figure \ref{fig:GraphConnected}. \cite{KM2008} states that $G_I$ is a connected bipartite graph of maximum degree three with $|V(G_I)|=22m$, $|E(G_I)|=24m-1$, and $\nu(G_I)=\frac{|V(G_I)|}{2}=11m$. Define: $k=7m+K-1$. Theorem 1 from \cite{KM2008} proves:

\begin{theorem}
    \label{thm:MasterThesisTheorem} (\cite{KM2008}) For every instance $I=(X, C, K)$ of Max 2-SAT, there exists a truth assignment $\Tilde{\alpha}=(\alpha_1,...,\alpha_n)$, such that at least $K$ of clauses $C_1, ..., C_m$ are satisfied by $\Tilde{\alpha}$, if and only if $L(G_I)\geq k$.
\end{theorem} Since the graph $G_I$ can be constructed from $I$ in polynomial time, Theorem \ref{thm:MasterThesisTheorem} implies the NP-hardness of computing $L(G)$ in connected, bipartite graphs of maximum degree three having a perfect matching.

\begin{figure}[ht]
\centering
  
  \begin{center}

		\begin{tikzpicture}[scale = 0.6]
			
			
			
			


               \node at (8,4) {Connecting to one};
               \node at (8,3) {of the two vertices};
               \node at (8,2) {$u_{11}$ of $G(C_1)$};

               \node at (8,9) {Connecting to one};
               \node at (8,8) {of the two vertices};
               \node at (8,7) {$u_{11}$ of $G(C_m)$};

                \draw[->, thick] (-2, 0) -- (10,0);
                \draw[->, thick] (0, -1) -- (0, 10);

                \node at (-1.35,-1) {$-1$};

                \node at (1,1) {$1$};
                \node at (1,2) {$2$};
                \node at (1,3) {$3$};
                \node at (1,4) {$4$};
                \node at (1.35,8) {$4m-1$};
                \node at (1,9) {$4m$};

                \draw [dashed] (-1,1) -- (0.5,1);
                \draw [dashed] (-1,2) -- (0.5,2);
                \draw [dashed] (-1,3) -- (0.5,3);
                \draw [dashed] (-1,4) -- (0.5,4);

                 \draw [dashed] (-1,8) -- (0.5,8);
                  \draw [dashed] (-1,9) -- (0.5,9);

                   \draw [dashed] (-1,1) -- (-1,0);

                    \draw [dashed] (-1,4) -- (-1,5);
                    \draw [dashed] (-1,8) -- (-1,7);



			\tikzstyle{every node}=[circle, draw, fill=black!50,inner sep=0pt, minimum width=4pt]
																							
			\node[circle,fill=black,draw] at (-1,1) (vm11) {};
                \node[circle,fill=black,draw] at (-1,2) (vm12) {};
                \node[circle,fill=black,draw] at (-1,3) (vm13) {};
                \node[circle,fill=black,draw] at (-1,4) (vm14) {};

                \node[circle,fill=black,draw] at (-1,8) (vm14minus1) {};
                \node[circle,fill=black,draw] at (-1,9) (vm14m) {};



			\path[every node]
			
			(vm11) edge (vm12)
                (vm12) edge (vm13)
                (vm13) edge (vm14)

                (vm14minus1) edge (vm14m)

                (vm14) edge [bend left] (5,4)
                (vm14m) edge [bend left] (5,9)

			;
		\end{tikzpicture}
																
	\end{center}
								
	\caption{Making sure that the resulting graph is connected.}
	\label{fig:GraphConnected}
\end{figure}
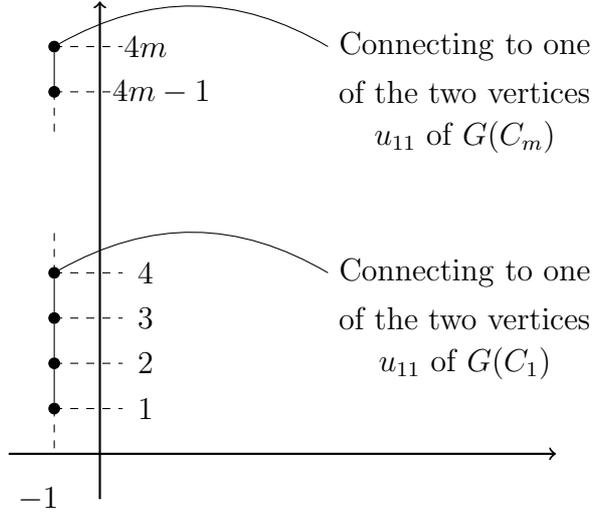

Now, we are going to introduce some constraints on vertices of $G_I$. Every vertex $z$ of $G(C_j)$ with degree one gets a constraint $W(z)=\{1\}$ (Figure \ref{fig:FourCycleGadget}). Moreover, the vertex $z=(-1,1)$ of $G_I$ (Figure \ref{fig:GraphConnected}) gets a constraint $W(z)=\{1\}$. For the remaining vertices $z$ we set $W(z)=\{1, 2\}$. Let $\nu_2^W(G)$ be the number of edges in the largest $2$-edge-colorable subgraph respecting our constraints $W(z)$ at every vertex $z$ of $G$. Note that if in a graph $H$, we have that at every vertex $x$ $W(x)=\{1,2\}$, then $\nu_2^W(H)=\nu_2(H)$.

\begin{lemma}
    \label{lem:nu2WFormula} For every instance $I=(X, C, K)$ of Max 2-SAT, we have \[\nu_2^W(G_I)=\nu(G_I)+L(G_I).\]
\end{lemma}

\begin{proof} We start by proving 
\[\nu_2^W(G_I)\geq \nu(G_I)+L(G_I).\]
Let $F$ be a perfect matching of $G_I$ with $L(G_I)=\nu(G\backslash F)$. Since $F$ is a perfect matching, we have that $F$ covers the vertex $(-1,1)$ (Figure \ref{fig:GraphConnected}) and all degree one vertices of graphs $G(C_j)$, $1\leq j\leq m$. Let $F'$ be a maximum matching in $G\backslash F$. If we color the edges of $F$ with color 1, and the edges of $F'$ with color 2, then note that this coloring will respect our constraints $W(z)$. Thus, we get the lower bound above.

Now, in order to prove the upper bound 
\[\nu_2^W(G_I)\leq \nu(G_I)+L(G_I),\]
we show that there is a largest 2-edge-colorable subgraph in $G=G_I$ respecting constraints $W$ such that edges of color 1 form a perfect matching.

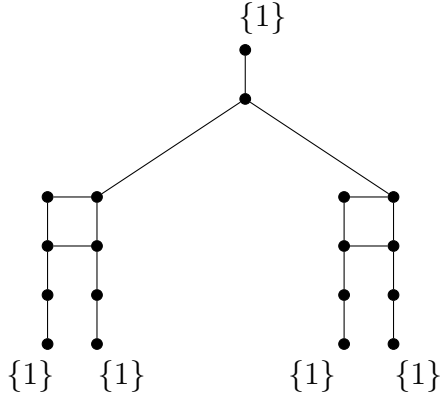
\begin{figure}[htbp]
 \begin{center}
  \begin{tikzpicture}[scale=0.65]
  
  \node at (4.35,4.65) {$\{1\}$};

  \node at (-0.35,-2.65) {$\{1\}$};
  \node at (1.5,-2.65) {$\{1\}$};

   \node at (5.35,-2.65) {$\{1\}$};
  \node at (7.5,-2.65) {$\{1\}$};

  

  \tikzstyle{every node}=[circle, draw, fill=black!50,
                        inner sep=0pt, minimum width=4pt]
                        
    \node[circle,fill=black,draw] at (0,0) (n00) {}; 
    \node[circle,fill=black,draw] at (0,1) (n01) {}; 
     \node[circle,fill=black,draw] at (0,-1) (n0m1) {};
     \node[circle,fill=black,draw] at (0,-2) (n0m2) {};

    \node[circle,fill=black,draw] at (1,1) (n11) {}; 
     \node[circle,fill=black,draw] at (1,0) (n10) {}; 
     \node[circle,fill=black,draw] at (1,-1) (n1m1) {};
     \node[circle,fill=black,draw] at (1,-2) (n1m2) {};

      \node[circle,fill=black,draw] at (6,0) (n60) {}; 
    \node[circle,fill=black,draw] at (6,1) (n61) {}; 
     \node[circle,fill=black,draw] at (6,-1) (n6m1) {};
     \node[circle,fill=black,draw] at (6,-2) (n6m2) {};

    \node[circle,fill=black,draw] at (7,1) (n71) {}; 
     \node[circle,fill=black,draw] at (7,0) (n70) {}; 
     \node[circle,fill=black,draw] at (7,-1) (n7m1) {};
     \node[circle,fill=black,draw] at (7,-2) (n7m2) {};
     \node[circle,fill=black,draw] at (4,3) (n43) {}; 
     \node[circle,fill=black,draw] at (4,4) (n44) {};

    \path[every node]
            
            (n00) edge (n0m1)
            (n0m1) edge (n0m2)

            (n10) edge (n1m1)
            (n1m1) edge (n1m2)

            (n00) edge (n01)
            (n01) edge (n11)
            (n11) edge (n10)
            (n10) edge (n00)


            (n60) edge (n6m1)
            (n6m1) edge (n6m2)

            (n70) edge (n7m1)
            (n7m1) edge (n7m2)

            (n60) edge (n61)
            (n61) edge (n71)
            (n71) edge (n70)
            (n70) edge (n60)
            (n43) edge (n44)
            (n43) edge (n11)
            (n43) edge (n71)

			;
   \end{tikzpicture}
 \end{center}
\caption{Degree one vertices $z$ of $G(C_j)$ get a constraint $W(z)=\{1\}$. Others have $W(z)=\{1,2\}$. Here we assume that $C_j$ is a disjunction of two variables.}
\label{fig:FourCycleGadget}       
\end{figure} In other words, if $(M_1, M_2)$ is a pair of matchings in $G$ with $|M_1|+|M_2|=\nu_2^W(G)$, then we can choose a pair, such that $M_1$ is a perfect matching. Note that we can always assume that degree one vertices are covered by $M_1$. This in particular implies that $M_1$ covers the degree one vertex $(-1,1)$, hence the edge incident to it is in $M_1$. By a similar reasoning, we have that we can assume that the edges of $M_1$ lying on the path $(-1,1)-(-1,4m)$ (Figure \ref{fig:GraphConnected}) form a perfect matching of this path. This in particular implies that the other edges of this path lie in $M_2$. Hence all edges of $G$ connecting a vertex of this path to a graph $G(C_j)$ (Figure \ref{fig:GraphConnected}) do not belong to $M_1\cup M_2$.

Since degree one vertices $z$ of $G(C_j)$ are assigned the color constraint $W(z)=\{1\}$, we can always assume that the edges incident to them are a subset of $M_1$. Hence, the edges of $G_I$ adjacent to edges $u_{11}u_{12}$, $u_{21}u_{22}$ are from $M_2$. Thus, what we are left is to show that $M_1$ covers all vertices of the cycles of length multiple to four (Figure \ref{fig:GraphCorrespondingVariable}).

Now, let us consider all pairs $(M_1, M_2)$ of matchings of $G$ with $\nu_2^W(G)=|M_1|+|M_2|$, that satisfy previous constraints, and among them choose one such that
\begin{itemize}
    \item edges adjacent to degree one vertices are a subset of $M_1$,

    \item subject to the previous condition, we choose a pair in which the number of uncovered vertices with respect to $M_1$ is minimized,

    \item subject to the previous condition, the shortest $M_1$-augmenting path has smallest length.
\end{itemize} We will show that the number of uncovered vertices is zero, hence $M_1$ is a perfect matching. Suppose there is an uncovered vertex. Since the cycles from Figure \ref{fig:GraphCorrespondingVariable} are even, we have that there is another uncovered vertex on it. Let us consider the shortest possible $M_1$-augmenting path $P$ that connects the uncovered vertices $u$ and $v$. If the first or last edge of $P$ is not from $M_2$, then we can flip the edge from $M_1$ adjacent to this edge and get a new pair $(M'_1, M'_2)$ of matchings such that $\nu_2^W(G)=|M'_1|+| M'_2|$, the pair $(M'_1, M'_2)$ satisfies the first two conditions, however with respect to it there is a shorter $M'_1$-augmenting path. Thus, w.l.o.g., we can assume that these two vertices are adjacent to an edge from $M_2$, hence they have color 2. This means that they are not an edge from the graphs $G(C_j)$ (Figure \ref{fig:GraphCorrespondingClause}) and have been added thanks to cyclical joining from Figure \ref{fig:GraphCorrespondingVariable}.

Now, consider the $M_1$-augmenting path $P_u$ starting from $u$ and moving towards the other edge (lying outside $M_2$, and $M_1$) of the long even cycle. Note that since pendant edges are from $M_1$, we assumed that the edges adjacent to them are from $M_2$ in our pair $(M_1, M_2)$ with $\nu_2^W(G)=|M_1|+|M_2|$. Note that the structure of our graphs $G=G_I$ implies that there is no edge of color 2 in this $M_1$-augmenting path $P_u$. This just follows from the observation that all edges of this path that are not an edge from the graphs $G(C_j)$ (Figure \ref{fig:GraphCorrespondingClause}) and have been added thanks to cyclical joining Figure \ref{fig:GraphCorrespondingVariable} must belong to $M_1$. The remaining edges of $P_u$ are adjacent to an edge from $M_2$. Thus, by augmenting $M_1$ on this path, we will obtain a new pair $(M''_1, M''_2)$, that respects our constraints $W$, and contains more edges than $(M_1, M_2)$ does and has less uncovered $M_1$ vertices. This is a contradiction.
\end{proof}

\begin{theorem}
    \label{thm:UnweightednuWHardness} The problem of computing $\nu_2^W(G)$ is NP-hard in the class of connected, bipartite graphs of maximum degree three.
\end{theorem}

\begin{proof} The proof follows from the NP-hardness of Max 2-SAT, Theorem \ref{thm:MasterThesisTheorem} and Lemma \ref{lem:nu2WFormula}. If $I=(X, C, K)$ is an instance of Max 2-SAT, then by Theorem \ref{thm:MasterThesisTheorem} and Lemma \ref{lem:nu2WFormula}, there is a truth assignment $\Tilde{\alpha}=(\alpha_1,...,\alpha_n)$, such that at least $K$ of clauses $C_1, ..., C_m$ are satisfied by $\Tilde{\alpha}$, if and only if $L(G_I)\geq k$, which is equivalent to $\nu_2^W(G_I)\geq \frac{|V(G_I)|}{2}+k$ since 
\[\nu(G_I)=\frac{|V(G_I)|}{2}\text{ and }\nu_2^W(G_I)=\nu(G_I)+L(G_I).\]
\end{proof}

\section{Conclusion and future work}
\label{sec:conclusion}

In this paper, we considered the maximum 2-edge-colorable subgraph problem in bipartite graphs. It is polynomial-time solvable in this class (as we observe in the beginning of the paper). A version of it where the edges of the input graph are unweighted however around every vertex we have a set of color constraints that shows which color can appear on edges incident to it is NP-hard in connected bipartite graphs of maximum degree three.

Lemma \ref{lem:ForestConstraintsWeightedKmatchings} implies that this result is not true in forests. The running time of the algorithm in this lemma has single exponential dependence on $k$. In terms of parameterized complexity theory \cite{FPTbook}, the result of the present paper demonstrates the so-called ``paraNP-hardness" of this problem when it is parameterized with respect to $k$. See \cite{kEdgeColoringFPT} for similar results.

The paper \cite{KowalikSIDMA:2018} states the problem of finding a $\chi'(G)$-coloring of an arbitrary graph $G$ in time $O^*(c^{n})$ as an interesting problem. Here $c>1$ and $n=|V|$. Note that this result would follow if the maximum $k$-edge-colorable subgraph problem can be solved in time $O^*(c^{n})$. We suspect that this is impossible, so an idea would be to demonstrate it under some complexity theoretical assumption like Exponential Time Hypothesis \cite{fomin:2010}. See \cite{fomin:2010} for more results of this type.



\bibliographystyle{elsarticle-num}



\end{document}